\documentclass[11pt,final]{amsart}
\usepackage[utf8]{inputenc}
\usepackage[T1]{fontenc}
\usepackage[letterpaper,centering]{geometry}
\usepackage{lmodern}
\usepackage{eucal}
\usepackage{mathrsfs}
\usepackage{amsmath,amssymb,amsfonts,amscd}
\usepackage{microtype,xcolor,enumitem,booktabs,needspace}
\definecolor{darkblue}{rgb}{0,0,0.3}
\definecolor{darkgreen}{rgb}{0,0.4,0}
\usepackage[colorlinks=true,linkcolor=darkblue,citecolor=darkgreen,
 urlcolor=darkblue,unicode,pdfborder={0 0 0}]{hyperref}
\hypersetup{pdftitle={Campana strong approximation on generalized Markoff surfaces},
 pdfauthor={Dasheng Wei},
 pdfsubject={Campana points, Markoff surfaces and the Brauer--Manin obstruction}}
\setlist[enumerate]{label={\upshape(\arabic*)},topsep=.7ex,leftmargin=*}
\setlist[itemize]{leftmargin=*}
\theoremstyle{plain}
\newtheorem{thm}{Theorem}[section]
\newtheorem{lem}[thm]{Lemma}
\newtheorem{prop}[thm]{Proposition}
\newtheorem{cor}[thm]{Corollary}

\theoremstyle{definition}
\newtheorem{defn}[thm]{Definition}
\newtheorem{rmk}[thm]{Remark}

\numberwithin{equation}{section}
\newcommand{\PP}{\mathbf P}

\newcommand{\A}{\mathbf A}
\newcommand{\OO}{\mathcal O}
\newcommand{\FF}{\mathbf F}
\newcommand{\ZZ}{\mathbf Z}
\newcommand{\QQ}{\mathbf Q}
\newcommand{\kbar}{\overline{k}}
\newcommand{\sX}{\mathscr X}
\newcommand{\sU}{\mathscr U}

\newcommand{\sD}{\mathscr D}

\newcommand{\sZ}{\mathscr Z}
\newcommand{\CC}{\mathcal C}
\newcommand{\Gm}{\mathbf G_{\mathrm m}}
\newcommand{\red}{\mathrm{red}}

\newcommand{\et}{\mathrm{\acute et}}
\DeclareMathOperator{\Br}{Br}

\DeclareMathOperator{\Spec}{Spec}

\DeclareMathOperator{\inv}{inv}
\DeclareMathOperator{\ord}{ord}
\newcommand{\sB}{\mathscr B}

\newcommand{\sV}{\mathscr V}

\title[Campana strong approximation]
{Campana strong approximation on generalized Markoff surfaces}
\author{Dasheng Wei}
\address{Academy of Mathematics and System Science, CAS, Beijing 100190,
P.\ R.\ China {\normalfont\itshape and} School of Mathematical Sciences,
University of CAS, Beijing 100049, P.\ R.\ China}
\email{dshwei@amss.ac.cn}
\date{}
\subjclass[2020]{Primary 14G05; Secondary 11G35, 14F22, 14J26}
\keywords{Campana points, generalized Markoff surfaces, Campana strong approximation,
Brauer--Manin obstruction}

\begin{document}

\begin{abstract}
For every smooth projective generalized Markoff surface $X_m$, every hyperplane
section $H$, and every integer $n\ge2$, we establish Campana strong
approximation with the Brauer--Manin obstruction for
$\bigl(X_m,(1-1/n)H\bigr)$.
\end{abstract}

\maketitle

\section{Introduction}

Let $k$ be a number field, fix $a,b,c\in k^\times$, and let
$m\in k\setminus\{0,4abc\}$. Consider the generalized Markoff surface
\begin{equation}\label{eq:projective}
 X_m:\quad (ax^2+by^2+cz^2)w-xyz=mw^3\subset\PP^3_k.
\end{equation}
Its boundary is the triangle $D=D_x+D_y+D_z$ with 
\begin{equation*}
 D_x=\{x=w=0\},\quad D_y=\{y=w=0\},\quad D_z=\{z=w=0\}. 
\end{equation*}
Writing $U_m=X_m\setminus D$, we have
\begin{equation}\label{eq:affine}
 U_m:\quad ax^2+by^2+cz^2-xyz=m.
\end{equation}
The surface $X_m$ is smooth when $m\notin\{0,4abc\}$.
The fixed coefficients $a,b,c$ are
suppressed from the notation $X_m$ and $U_m$. In the special case $a=b=c=1$, the
geometry and the Brauer group were studied in \cite {CTWX,LM}.

The classical Markoff equation \(x^2+y^2+z^2=3xyz\) first appeared in 1879 \cite{Markoff}. Its positive integral points carry an action of a discrete group and are generated by $(1,1,1)$.  Hurwitz studied its higher-dimensional analogues \cite{Hurwitz}.
Much subsequent work explores connections between these equations, counting problems, dynamical systems, and local‑global principles. Important results include strong approximation modulo primes \cite{BGS}, asymptotic counting for integral points \cite{GMR}, and relations to Painlevé dynamics on character varieties \cite{CL}.

For the family \(x^2+y^2+z^2-xyz=m\), Ghosh and Sarnak \cite{GS} developed reduction theory for integral solutions. They proved the integral Hasse principle
for almost all integer parameters and exhibited infinitely many
failures. Loughran and Mitankin \cite{LM}, together with Colliot-Th\'el\`ene, Xu, and the author \cite{CTWX}, computed the full Brauer group and further studied failures of the integral Hasse principle via the Brauer--Manin obstruction. Some twisted and number-field generalizations have been treated in \cite{Chen, Dao, KSS}.

Campana points \cite{Campana} interpolate between rational and
integral points. For a finite weight $n$, a positive boundary
multiplicity is permitted only when it is at least $n$. If $n$ is infinite, then Campana points recover $S$-integral points on the complement. 
Here we consider finite weights.
With the basic notion of Campana points in place, numerous related arithmetic results have been obtained.
A Manin-type conjecture and counting results for Campana points are
developed in \cite{PSTVA}, with further counting results in \cite{BY,PS}. Turning to approximation theory,
weak approximation for Campana points has been studied for projective spaces and quadrics in
\cite{NS,MNS,Wei26}. Failures of local--global principles for Markoff orbifold
pairs were studied by Mitankin and Uhlemann \cite{MU}.

We now return to the generalized surface over a number field $k$. Let $\Pi$ be a hyperplane in $\PP^3$,
and let $H=X_m\cap\Pi$ be the corresponding (reduced) hyperplane section. Set
$\Delta_H=(1-1/n)H$, where $2\le n<\infty$.
Fix a finite set $S\supseteq\Omega_k^\infty$ and a proper flat
integral model $\sX_m$ over $\OO_S$. In this paper, we follow
\cite[Definitions~2.2--2.4]{NS} for the local intersection multiplicity. 
Let $\CC_v(\sX_m,\Delta_H)$ be the local Campana space.
We equip
\[
 (X_m,\Delta_H)(\A_k)
 =\prod_v'\bigl(X_m(k_v),\CC_v(\sX_m,\Delta_H)\bigr)
\]
with the restricted-product topology, which we call the
\emph{Campana topology}. 

\begin{thm}\label{thm:main}
Let $a,b,c\in k^\times$, let $m\in k\setminus\{0,4abc\}$, and let
$X_m$ be the surface in \eqref{eq:projective}.
For every hyperplane section $H$ of $X_m$ and every integer
$n\ge2$, Campana strong approximation with the Brauer--Manin
obstruction holds for $\bigl(X_m,(1-1/n)H\bigr)$; that is,
\begin{equation*}
 \overline{(X_m\setminus H)(k)}
 =(X_m,\Delta_H)(\A_k)^{\Br(X_m)}.
\end{equation*}
The closure is taken in the Campana topology.
\end{thm}

Ordinary strong approximation behaves quite differently. Suppose that $k=\QQ$ and $a=b=c=1$, and let $U_m$ be the affine
variety defined by \eqref{eq:affine}.
Weak strong approximation always fails on $U_m$ \cite[Theorem~1.1]{CTWX}; hence
strong approximation with the Brauer--Manin obstruction off the infinite place also always fails on $U_m$
\cite[Corollary~1.2]{CTWX}. Thus, Theorem~\ref{thm:main} does not hold when $n=\infty$.

Any $k$-line induces a fibration of $X_m$ over $\PP^1$. Such fibrations play a key role in the proof of this theorem.  
Its proof splits into two parts. If $H\ne D$, we may choose two sides of $D$, which are
a pair of intersecting $k$-lines. Proposition~\ref{prop:twolines}
therefore settles this case directly by the fibration method.

 Let $f$ be the fibration defined by $D_x$. If $H=D$, then the fibre $f^{-1}(\infty)=D_y+D_z$ lies in the boundary, which may be regarded as having multiplicity $n$.
The corresponding fibration problem is therefore particularly difficult, since it has one multiple fibre and four additional nonsplit fibres. 
Instead, we use the action of Vieta involutions, which allows us to move any local point with reduction on $D$ arbitrarily close to the vertices of $D$; this observation allows us to complete the proof. The three vertices of $D$ are essential to the argument, so this method no longer applies if they are removed from $X_m$, although one may expect similar approximation properties after removing them, by arithmetic purity. 

In fact, our result applies more broadly. For a smooth cubic surface containing two distinct $k$-lines, Campana strong approximation with the Brauer--Manin
obstruction holds if $H$ does not contain both of these
lines; see Proposition~\ref{prop:twolines}. If the two lines are disjoint, then $H$ automatically fails to contain both. Hence,
Campana strong approximation holds for every $H$. Here the Brauer--Manin obstruction vanishes, since the surface is $k$-rational.  For a smooth
cubic surface containing a $k$-line $\ell$, we also prove the same result if $H$ avoids the nodes of all singular
fibres and $\ell\not \subset H$; see
Proposition~\ref{prop:cubic}. This shows that the approximation property holds for any $H$ in general position.

Once
Theorem~\ref{thm:main} has been established, its equal-weight boundary
case yields the following unequal-weight corollary.

\begin{cor}\label{cor:boundary}
For any integers $n_x,n_y,n_z\ge2$, let
\begin{equation}\label{eq:orbifold}
 \Delta_{\boldsymbol n}
 =\left(1-\frac1{n_x}\right)D_x+
 \left(1-\frac1{n_y}\right)D_y+
 \left(1-\frac1{n_z}\right)D_z.
\end{equation}
Campana strong approximation with the Brauer--Manin obstruction holds
for $(X_m,\Delta_{\boldsymbol n})$. 
\end{cor}




The paper is organised as follows. In Section~\ref{sec:definitions} we define the Campana topology and Campana strong approximation, and we also establish some basic properties of this topology.
In Section~\ref{sec:criterion} we establish a two-conic-bundle
criterion by the fibration method.  Applying it to cubic surfaces, we treat the hyperplane sections of $X_m$ other than $D$. 
Section~\ref{sec:vieta} treats the hardest case, that of the boundary triangle, using Vieta
involutions.

\noindent\textbf{Acknowledgements.}
This work is supported by the National Natural Science Foundation
of China (Grant Nos.~12371014 and 12231009).
\section{Notation and preliminaries}\label{sec:definitions}

\subsection{Campana topology and strong approximation}\label{subsec:notation}


Let $k$ be a number field with a fixed algebraic closure $\kbar$. Denote by $\Omega_k$ the set of all places of $k$, by $\Omega_k^\infty\subset\Omega_k$ its archimedean places, and by $\mathbb A_k$ its ad\`ele ring. For each place $v$, let $k_v$ be the completion of $k$ at $v$. When $v$ is finite, let $\mathcal O_v$ be its valuation ring with the residue field $\mathbb F_v$,  $\pi_v$ a uniformizer, and  $\ord_v$ the normalized valuation. We set $\ord_v(0)=\infty$.
The finite set $S\subset \Omega_k$ is always assumed to contain $\Omega_k^\infty$, and we write $\mathcal O_S$ for the ring of $S$-integers. 

Given a $k$-variety $Y$, we write $\overline Y=Y\times_k\kbar$. Its Brauer group is
$\Br(Y)=H^2_{\et}(Y,\Gm)$,  and its algebraic Brauer group is $\Br_1(Y)=\ker\bigl(\Br(Y)\longrightarrow\Br(\overline Y)\bigr)$.
Let $\Br_0(Y)\subseteq\Br_1(Y)$ be the constant Brauer subgroup.

For a smooth $k$-variety $Y$ and a subgroup $B\subseteq\Br(Y)$, set
\[
 Y(\A_k)^B=
 \left\{(P_v)\in Y(\A_k):
 \sum_{v\in\Omega_k}\inv_v\alpha(P_v)=0
 \text{ for every }\alpha\in B\right\}.
\]
Here $Y(\A_k)$ carries its usual adelic topology; 
if
$Y$ is proper, this topology coincides with  the
product topology. 
Sometimes we abbreviate $Y(\A_k)^{\Br(Y)}$ as
$Y(\A_k)^{\Br}$. 


Let $X$ be a smooth, separated, geometrically integral $k$-variety,
and let
\[
 \Delta=\sum_{i\in I}\left(1-\frac1{n_i}\right)B_i,
 \qquad n_i\in\ZZ_{\ge2}\cup\{\infty\},
\]
where the $B_i$ are distinct prime divisors of $X$ and we set
$1/\infty=0$.

 Let $\sX$ be a separated flat
$\OO_S$-model of $X$, let $\sB_i$ be the
Zariski closure of $B_i$ in $\sX$, and let
\[
 \sD=\sum_{i\in I}\left(1-\frac1{n_i}\right)\sB_i.
\]
Then $(\sX,\sD)$ is called an $\OO_S$-model of $(X,\Delta)$.
Here $X$ is not assumed to be proper.

For any local integral point $x\in\sX(\OO_v)$, let
$\widetilde x:\Spec\OO_v\to\sX_{\OO_v}$ denote the corresponding
section.
For a closed subscheme
$\mathscr E\subset\sX_{\OO_v}$ with ideal sheaf
$\mathcal I_{\mathscr E}$, set
$J_{\mathscr E,x}=\widetilde x^{-1}\mathcal I_{\mathscr E}\cdot\OO_v$
and define the intersection multiplicity 
\begin{equation*}
 n_v(\mathscr E,x)=
 \begin{cases}
 \operatorname{length}_{\OO_v}(\OO_v/J_{\mathscr E,x}),
 &J_{\mathscr E,x}\ne0,\\
 \infty,&J_{\mathscr E,x}=0.
 \end{cases}
\end{equation*}
Thus, when the multiplicity is finite,
$J_{\mathscr E,x}=(\pi_v^{\,n_v(\mathscr E,x)})$.
If $\mathscr E$ is locally generated  by a function $f$ near $x$,
then we have $n_v(\mathscr E,x)=\ord_v(f(x))$. 

For any place $v\notin S$, we write
\[
B_i\times_k k_v=\bigcup_{\lambda\in I_i(v)}B_{i,\lambda},
\]
where the $B_{i,\lambda}$ are the $k_v$-irreducible components of $B_i$.
Let $\sB_{i,\lambda}$ denote their closures in
$\sX_{\OO_v}$. 
For $2\le n<\infty$, put
$\mathcal N_n=\{0\}\cup\ZZ_{\ge n}\cup\{\infty\}$, and set
$\mathcal N_{\infty}=\{0\}$. Let
$B=\sum_{i\in I}B_i$ and let $ X^\circ=X\setminus B$.
Following the local-component convention of \cite[Definition~2.4]{NS}, for any finite $S$, we set
\begin{align*}
 \CC_v(\sX,\Delta)&=\{x\in\sX(\OO_v):
 n_v(\sB_{i,\lambda},x)\in\mathcal N_{n_i}
 \text{ for all }i \in I\text { and }\lambda\in I_i(v)\}, &&v\notin S,\\
 \CC_v(\sX,\Delta)&=X(k_v),&&v\in S.
\end{align*}
These local Campana inequalities are imposed on each $k_v$-irreducible component separately, unlike the conditions in \cite[Definition~3.4]{PSTVA}, which are imposed on the full closure of each prime boundary $k$-divisor.
For any $\OO_S$-model $\sX$ and finite  $T\supseteq S$, put
\begin{equation*}
\mathcal A_T(\sX,\Delta)=\prod_{v\in T}X(k_v)\times\prod_{v\notin T}\CC_v(\sX,\Delta)
\end{equation*}
with the product topology. 
 When $X$ and its model are proper,
$\sX(\OO_v)=X(k_v)$, so we recover the local spaces of~\cite{NS}.

Define the \emph{Campana adelic space} by
\begin{equation*}
 (X,\Delta)(\A_k)
 =\prod_v'\bigl(X(k_v),\CC_v(\sX,\Delta)\bigr)
 =\bigcup_{T\supseteq S\ {\rm finite}}\mathcal A_T(\sX,\Delta).
\end{equation*}
The restricted-product topology, which we call the
\emph{Campana topology}, has a basis consisting of
\begin{equation*}
 \prod_{v\in F}W_v\times\prod_{v\notin F}\CC_v(\sX,\Delta),
 \qquad F\supseteq S\text{ finite},
\end{equation*}
where each $W_v\subseteq X(k_v)$ is open. As with the adelic space $X(\A_k)$, the Campana adelic space and its topology do not depend on the choice of 
an $\OO_S$-model, since any two integral models become isomorphic after removing finitely many
places. By the Lang--Weil estimate, $\sX^\circ(\OO_v)\neq \emptyset$ for almost all $v$. Hence the Campana adelic space is  dense in $X(\A_k)$, and its topology is finer than the subspace topology inherited from $X(\A_k)$.
If the boundary $\Delta$ is empty, it recovers $X(\A_k)$. Every rational point  $x\in X(k)$ lies in the Campana adelic space if all weights are finite. In fact, if $x$ lies on the boundary, the intersection multiplicity is either  $0$ or $\infty$ outside finitely many places. If $x\in X^\circ(k)$, such an $x$ extends to the model
and avoids the boundary at almost all places. Points in $X^\circ$ are called \emph{strict points}.

For any finite $T\subset\Omega_k$, the projection
$\operatorname{pr}^{T}$ omits the components at the places in $T$.
The image $\operatorname{pr}^{T}\bigl((X,\Delta)(\A_k)\bigr)$ carries
the topology induced by the restricted product
\[
 \prod_{v\notin T}'\bigl(X(k_v),\CC_v(\sX,\Delta)\bigr).
\]
If $\prod_{v\in T}X(k_v)\ne\varnothing$,  the projected image agrees with this restricted product.
Denote
\[
(X,\Delta)(\A_k)^{\Br(X)}
=(X,\Delta)(\A_k)\cap X(\A_k)^{\Br(X)}.
\]

\begin{defn}\label{def:strong} 
We say that $(X,\Delta)$ satisfies: 
\begin{itemize} 

	\item \emph{Campana strong
		approximation} if 
		$$\overline{X^\circ(k)}
		=(X,\Delta)(\A_k);$$
	\item 
	\emph{Campana strong approximation with the Brauer--Manin obstruction} if 
	$$\overline{X^{\circ}(k)}
	=(X,\Delta)(\A_k)^{\Br};$$
	\item \emph{Campana strong
	approximation off $T$} if 
$$\overline{\operatorname{pr}^{T}(X^\circ(k))}
=\operatorname{pr}^{T}((X,\Delta)(\A_k));$$
\item 
\emph{Campana strong approximation with the Brauer--Manin obstruction off $T$} if 
$$\overline{\operatorname{pr}^{T}(X^\circ(k))}
=\operatorname{pr}^{T}((X,\Delta)(\A_k)^{\Br}),$$
\end{itemize}
where each closure is taken with respect to the topology on the right-hand side.
\end{defn}


Our notion of Campana strong approximation differs from the strong approximation property in
	\cite[Definitions 2.15 and 2.23]{MNS}. 
In \cite{MNS}, strong approximation is defined for a fixed model.  The
restricted-product topology uses reference subsets
$\sU(\OO_v)$ with $\sU=\sX\setminus\bigcup_i\sB_i$.
All neighborhoods therefore require avoiding the boundary outside a
finite set of places. However, our topology permits positive boundary
intersection.
In fact,  our notion of Campana strong approximation
is equivalent to Campana weak approximation for every finite 
$S$ and every $\OO_S$-model; see \cite[Definition~2.7(ii)]{NS}.
Finally, the strict Brauer--Manin condition in \mbox{
	\cite[Definition~3.7]{MNS}
}\hskip0pt
is defined using $\Br(X^\circ)$, while we use $\Br(X)$.

If all weights are infinite, 
then $\CC_v(\sX,\Delta)=\sX^\circ(\OO_v)$ for $v\notin S$.
Campana strong approximation for $(X,\Delta)$ is
equivalent to ordinary strong approximation for
$X^\circ$.

%

The following statements are contained in \cite[Propositions~3.15 and~4.6]{Moerman}.

\begin{lem}\label{lem:open}
The space $\CC_v(\sX,\Delta)$ is open and closed in $X(k_v)$. 
For any dense open subset  $V$ of $X$, $\CC_v(\sX,\Delta)\cap V(k_v)$ is dense in $\CC_v(\sX,\Delta)$.
\end{lem}

\begin{proof} 
For $v\in S$, the set is simply $X(k_v)$, so no further argument is needed. If $v\notin S$,
for any $s\in \sX(\OO_v)$, the reduction $s \mod 
\pi_v^n$ is locally constant. For a finite weight $n_i$, failure of
the Campana condition means that the intersection multiplicity lies in
$\{1,\ldots,n_i-1\}$; for an infinite weight, we require that $\bar s$
does not lie in the boundary. Therefore, the local
Campana set is open and closed. By the implicit function theorem \cite[Proposition 3.5.73]{Poonen}, $\CC_v(\sX,\Delta)\cap V(k_v)$ is dense in $\CC_v(\sX,\Delta)$.
\end{proof}


\needspace{16\baselineskip}
The following comparison is
contained in
\cite[Propositions~4.12 and~4.13]{Moerman}.

\begin{prop}\label{prop:equivalence}
Campana strong approximation for $(X,\Delta)$ is equivalent to
Campana weak approximation for every finite
$S\supseteq\Omega_k^\infty$ and every $\OO_S$-model of $(X,\Delta)$;
see \cite[Definition~2.7(ii)]{NS}. The same equivalence holds
after imposing the Brauer--Manin obstruction throughout.
\end{prop}

\begin{proof} First, suppose that Campana strong approximation holds for $(X,\Delta)$.
For an $\OO_S$-model, $\mathcal A_S(\sX,\Delta)$ is an open subset of $(X,\Delta)(\A_k)$, so the density of $X^\circ(k)$ in $(X,\Delta)(\A_k)$ implies density in $\mathcal A_S(\sX,\Delta)$.  Conversely, given an open subset of $(X,\Delta)(\A_k)$, without loss of generality, we may assume  that it has the form $\prod_{v\in S}W_v\times\prod_{v\notin S}\CC_v(\sX,\Delta)$ for a given $\OO_S$-model. This is an open subset of $\mathcal A_S(\sX,\Delta)$. Hence Campana weak approximation for $S$ and this $\OO_S$-model yields a rational point of $X^\circ$ in this open subset, proving Campana strong approximation for $(X,\Delta)$. The same argument applies after intersecting with
the Brauer–Manin set.
\end{proof}

Let $\Delta_1\ge\Delta_2$ be Campana divisors on $X$ with weights in
$\ZZ_{\ge2}\cup\{\infty\}$. For a given $\OO_S$-model,
the local Campana conditions immediately give
\[
 \CC_v(\sX,\Delta_1)\subseteq\CC_v(\sX,\Delta_2).
\]

The argument in part~(1) below appears in
\cite[Proposition~4.5(1)]{Moerman}.

\begin{lem}\label{lem:monotonicity}
Let $T\subset\Omega_k$ be finite, possibly empty.
\begin{enumerate}
\item Campana strong approximation off $T$ for $(X,\Delta_1)$
implies Campana strong approximation off $T$ for
$(X,\Delta_2)$.
\item Assume that $\Br(X)/\Br_0(X)$ is finite.
Campana strong approximation with the Brauer--Manin obstruction
off $T$ for $(X,\Delta_1)$ implies the corresponding property
for $(X,\Delta_2)$.
\end{enumerate}
\end{lem}

\begin{proof}
The inclusion
$(X,\Delta_1)(\A_k)\hookrightarrow
(X,\Delta_2)(\A_k)$ is continuous and dense, and the same holds for the spaces away from $T$. Assertion (1) then follows from continuity.

For~(2), we only need to show the density of $(X,\Delta_1)(\A_k)^{\Br}$  in $(X,\Delta_2)(\A_k)^{\Br}$. We fix an integral $\OO_k$-model $\sX$ of $X$.
Choose a finite subset $B\subset \Br(X)$ that surjects onto
$\Br(X)/\Br_0(X)$.  Enlarge the finite set $S$ so that the elements of $B$ extend to $\Br(\sX_{\OO_S})$. Then
the evaluation of $B$ at $x_v$ vanishes for every $v \notin S$ and every $x_v\in \sX(\OO_v)$. For any point $(x_v)\in (X,\Delta_2)(\A_k)^{\Br}$, we can modify $x_v$ at all $v$ outside a sufficiently large finite set so that $x_v\in \CC_v(\sX,\Delta_1)$ and $(x_v)\in X(\A_k)^{\Br}$, i.e., $(x_v)\in (X,\Delta_1)(\A_k)^{\Br}$. This proves the density of $(X,\Delta_1)(\A_k)^{\Br}$  in $(X,\Delta_2)(\A_k)^{\Br}$, completing the proof.
%
%
%
\end{proof}



\subsection{Approximation on rational conics}

\begin{prop}\label{prop:curve}
Let $C$ be a smooth conic with a $k$-point, and let $E\subset C$ be
a reduced effective divisor of degree at most two. For  any assignment of finite weights to the closed points of $E$, the resulting pair satisfies
Campana strong approximation.
\end{prop}

\begin{proof}
Identify $C$ with $\PP^1_k$. If $E$ is empty, the assertion reduces
to ordinary weak approximation. Otherwise, a reduced divisor of degree
at most two is supported either on at most two $k$-points or on a single
closed point of degree two. Campana weak approximation for every $S$ and
every model follows from \cite[Propositions~3.5 and~3.15]{NS} (or from \cite[Theorem~1.1]{Wei26}). The claim now follows from Proposition~\ref{prop:equivalence}.
\end{proof}


\begin{lem}\label{lem:transfer}
Suppose $(X,\Delta)$ has finite weights. Let $W$ be an open subset of $(X,\Delta)(\A_k)$. Suppose that there exists a smooth rational conic $C\subset X$, not contained in $B$, satisfying the following conditions: 
\begin{itemize}
	\item $(C\cap B)_{\red}$ has degree at most two;
	\item $C(\A_k)\cap W \neq \emptyset $.
\end{itemize}

Then there is a
strict rational point of $C$ lying in $W$.
\end{lem}
\begin{proof}
	To each closed point of $C\cap B$, we assign the maximal weight among the components of $B$ containing it, and we denote the resulting boundary on $C$ by $E$. It is clear that the natural map $(C,E)(\A_k) \to (X,\Delta)(\A_k)$ is continuous. Hence $C(\A_k)\cap W$ contains a nonempty open subset of $(C,E)(\A_k)$, and the assertion follows from Proposition~\ref{prop:curve}.
%
\end{proof}

\section{A criterion for conic bundles and cubic surfaces}
\label{sec:criterion}

Let $X$ be a smooth, proper, geometrically rational surface with a conic fibration
$f:X\to\PP^1_k$. Assume that every geometrically
singular fibre consists of two distinct lines. 
Let $Z_f$ be the
finite  reduced scheme of their nodes.  

Let $H$ be a reduced effective divisor. Fix a proper $\OO_S$-model $\sX$ of $X$, and let 
$\sZ_f$ be the closure of $Z_f$. We always enlarge $S$ so that:
\begin{itemize}
		\item the model
	$\sX_{\OO_S}$ is smooth over $\OO_S$, 
	\item  $f$ extends to an $\OO_S$-morphism
	$\widetilde f:\sX_{\OO_S}\to\PP^1_{\OO_S}$,
	\item 
	$\sZ_f$ is finite \'etale over $\OO_S$, 
	\item  the reductions of  singular fibres of $\widetilde f$  are pairs of distinct lines
	whose nodes are  the reductions of $Z_f$, 
	\item if $H$ contains  no fibre of $f$,  no fibre is contained in the closure of $H$,
	\item if $H\cap Z_f=\varnothing$, their reductions do not meet,
	\item if a split fibre component is not contained in $H$, then it contains a smooth $\FF_v$-point outside~$H$. 
\end{itemize}

\begin{lem}\label{lem:localfibres}
Let $H$ be a reduced effective divisor containing no fibre of $f$,
and assume that its restriction to the generic fibre is reduced of
degree at most two. 
Let $C$ be  
a smooth $k$-fibre  containing a point $R\in C(k)$. For any $v\notin S$,  if
\[
\bar R_v\notin\sZ_f(\FF_v)
\qquad\text{or}\qquad H\cap Z_f=\varnothing,
\]
then there exists $P_v\in C(k_v)$ whose reduction lies outside $H$.
%
\end{lem}

\begin{proof}

The special fibre of the closure of $C$ is either smooth or a pair of
distinct lines. In the smooth case, it has a rational point outside $H$;
in the split case, a component not contained in $H$ has a smooth
rational point outside $H$. These points lift to the chosen fibre by
Hensel's lemma, applied where $f$ is smooth.

A non-split singular conic over the residue field has only one rational
point, namely its node. Hence the reduction of $R$ must be the node.
This is impossible when $R$ avoids $Z_f$. If instead $H$ avoids $Z_f$,
then the section defined by $R$ itself avoids $H$. This covers all cases.
\end{proof}

Suppose that  $X$ has another conic fibration $g$ over $\PP^1_k$. Let $A$ be the $g$-image of $Z_f$ in $\PP^1_k$.  Enlarge $S$ so that the corresponding properties stated for $f$ also hold for $g$, the closure  $\mathscr A $ of $A$ is \'etale over $\OO_S$ and $\widetilde g(\sZ_f)\subset\mathscr A$.

\begin{lem}\label{lem:twoavoid}
Suppose that $g:X\to\PP^1_k$ is a second conic bundle. Let $A=g(Z_f)\subset \PP^1_k$. For every $a\in A$, suppose that the fibre
$g^{-1}(a)$ is either smooth or a pair of lines defined over $k(a)$.
Then, for any $v\notin S$ 
and any $R_0\in X(k)$ lying on a
smooth $g$-fibre $C_s$ with $s\notin A$, there is a point
$R\in C_s(k)$ arbitrarily close to $R_0$ at the places in $S$, whose reductions avoid
$Z_f$ at all $v\notin S$.
\end{lem}

\begin{proof}
%
Let $T=\{v\notin\text{ S}: \bar s_v \in \mathscr A\}$,  which is a finite set because $s\notin A$.  For any $v\in T$, the
residue fibre is smooth or split, so we may choose a smooth $\FF_v$-point outside $\sZ_f$. On the fibre over~$\bar s$, Hensel's lemma then lifts this point to $Q_v\in C_s(k_v)$.

The smooth conic $C_s$ contains $R_0$, so $C_s\simeq\PP^1_k$.
Weak approximation on $C_s$ gives a point $R\in C_s(k)$ arbitrarily
close to $R_0$ at $v\in S$ and sufficiently close to \(Q_v\) that \(\bar R=\bar Q_v\) at $v\in T$. If $v\notin S\cup T$,  then the reduction $\bar s_v$ of $s$ lies outside
$\mathscr A$. Therefore, every point of $C_s(k_v)$ has reduction outside
$\sZ_f$, since $\widetilde g(\sZ_f)\subset\mathscr A$. 
The reduction of $R$ therefore avoids
$Z_f$ at every place outside $S$.
\end{proof}

\begin{thm}\label{thm:two}
Assume that $X(k)$ is dense in $X(\A_k)^{\Br(X)}$ and that $f$ and
$g$ satisfy the hypotheses of Lemma~\ref{lem:twoavoid}. Let $H$ be a
reduced effective divisor containing no fibre of $f$. Suppose that
the  restriction of $H$ to a general $f$-fibre is  of degree at most two. For every
integer $n\ge2$, Campana strong approximation with the Brauer--Manin
obstruction holds for $\bigl(X,(1-1/n)H\bigr)$.
\end{thm}

\begin{proof} Let $S$ be the finite set of exceptional places of Lemma~\ref{lem:twoavoid}. Let $\Delta=(1-1/n)H$. For any $(P_v)\in (X,\Delta)(\A_k)^{\Br(X)}$ and any open neighbourhood $\prod_{v\in T}W_v\times\prod_{v\notin T}\CC_v(\sX,\Delta)$ of $(P_v)$, where $T\supset S$ is finite, it suffices to find a strict rational point in this neighbourhood.

 Since $X(k)$ is dense in $X(\A_k)^{\Br(X)}$, we may choose  a point $R_0\in X(k)$ lying in $W_v$ for $v\in T$ such that
its $g$-fibre is smooth with parameter outside $A$ and its $f$-fibre is
smooth. Lemma~\ref{lem:twoavoid} gives a point $R$, arbitrarily close to
$R_0$ at $T$, whose reductions avoid $Z_f$ outside $T$. 
We may assume that the $f$-fibre $C$ containing $R$ is smooth and that its intersection with $H$ has degree at most two.
The conic $C$ contains $R$ and is therefore
rational. Lemma~\ref{lem:localfibres} gives a local point on $C$ whose
section avoids $H$ at every $v\notin T$.
Applying Lemma~\ref{lem:transfer} to $C$, we obtain a strict Campana point arbitrarily close to $R$, hence in the given neighbourhood. 
%
%
\end{proof}

Turning to cubic surfaces, first we have the following approximation property.
\begin{prop}\label{prop:ordinary}
	Let $X\subset\PP^3_k$ be a smooth cubic surface with a $k$-line
	$\ell$.
The rational points of $X$ are dense in
$X(\A_k)^{\Br(X)}$. 
\end{prop}

\begin{proof}
We contract the $k$-line $\ell$ to produce a smooth del Pezzo surface $Y$
of degree four; the image of $\ell$ is a $k$-point. 
By the theorem of Salberger and Skorobogatov \cite[Theorem~0.1]{SS},
$Y$ satisfies weak
approximation with the Brauer--Manin obstruction. By birational invariance \cite[p.~1990]{HS}, the same approximation property holds on $X$. 
\end{proof}

 On a  smooth cubic surface $X$, every
 hyperplane section $H$ is reduced.  Over $\kbar$, a nonreduced section would
 be $2M+M'$ or $3M$, where $M, M'$ are lines on $X$. If $H=2M+M'$, then $$H.M=2M^2+M'.M=-2+1=-1.$$ 
 If $H=3M$,  then similarly $H.M=3M^2=-3$.  Both lead to contradictions since the intersection number of a line with any hyperplane section is always~1. 
\begin{prop}\label{prop:cubic}
Let $X\subset\PP^3_k$ be a smooth cubic surface with a $k$-line
$\ell$, and let $f$ be the conic bundle given by planes through
$\ell$. Let $H$ be a hyperplane section not containing $\ell$.

If $H\cap Z_f=\varnothing$ or there is
a second bundle $g$ satisfying Lemma~\ref{lem:twoavoid}, then, for every integer $n\ge2$, Campana
strong approximation with the Brauer--Manin obstruction holds for
$\bigl(X,(1-1/n)H\bigr)$.
\end{prop}

\begin{proof}
Let $L$ be the hyperplane class. Each $f$-fibre has class
$L-\ell$, so $H\cdot(L-\ell)=3-1=2$. Therefore, the restriction
of $H$ to a general fibre has degree two.  The plane spanned by such a fibre
contains $\ell$, hence no $f$-fibre is contained in $H$. 



Let $S$ be the set of exceptional places of Lemma~\ref{lem:localfibres}. 
For any  adelic Campana 
point $(P_v)$ orthogonal to $\Br(X)$ and any open neighbourhood $\prod_{v\in T}W_v\times\prod_{v\notin T}\CC_v(\sX,(1-1/n)H)$, where $T\supset S$ is finite, Proposition~\ref{prop:ordinary} gives a rational point $R$ in $W_v$ for every $v\in T$ such that the
$f$-fibre $C$ containing $R$ is smooth and meets $H$ in a divisor of
degree two.


If $H\cap Z_f=\varnothing$, Lemma~\ref{lem:localfibres} supplies
local points on $C$ whose sections avoid $H$ at every $v\notin T$.
By Lemma~\ref{lem:transfer}, we may approximate $R$ at the places in $T$ to obtain 
a strict Campana point.
If a second bundle satisfies
Lemma~\ref{lem:twoavoid}, apply Theorem~\ref{thm:two} instead.
\end{proof}

\begin{prop}\label{prop:twolines}
	Let $X\subset\PP^3_k$ be a smooth cubic surface containing two distinct
	$k$-lines $\ell_1,\ell_2$. If a hyperplane section $H$ does not
	contain $\ell_1$, then, for every integer $n\ge2$, Campana strong
	approximation with the Brauer--Manin obstruction holds for
	$\bigl(X,(1-1/n)H\bigr)$.
	
	In particular, if the lines intersect, this holds for every hyperplane other than
	the plane containing both lines;
	if the lines are disjoint, Campana strong approximation holds for
	\emph{every} hyperplane section $H$, without the Brauer--Manin obstruction.
\end{prop}

\begin{proof} 
	Let $L$ be the hyperplane class, and let $f,g$ be the conic bundles
	associated with $\ell_1,\ell_2$.
	Suppose first that the lines are disjoint. Since no plane contains
	two disjoint lines,  we may assume $\ell_1\not\subset H$. It is clear that
	\[
	(L-\ell_1)\cdot\ell_2=(L-\ell_2)\cdot\ell_1=1,
	\]
	so $\ell_2$ and $\ell_1$ are sections of $f$ and $g$, respectively.
	A section $s$ of $g$ meets every fibre in a smooth point, since
	$dg\circ ds=\mathrm{id}$. Every singular $g$-fibre therefore has a
	smooth point over its residue field, hence its two components are defined
	over that field. The fibre hypothesis of
	Lemma~\ref{lem:twoavoid} holds, so
	Proposition~\ref{prop:cubic} yields the assertion.
	The section of $f$ shows that its generic conic is rational, so
	$k(X)\simeq k(t)(u)$ and $X$ is $k$-rational. Consequently,
	$\Br(X)=\Br_0(X)$ by birational invariance
	\cite[Corollary~6.8.7]{Poonen}.  

	Now suppose that the lines intersect.
	Write $\ell_3$ for the third line in their common plane; it is a
	$k$-line because the residual divisor has degree one. Let $a_0\in\PP^1(k)$ satisfy
	$g^{-1}(a_0)=\ell_1+\ell_3$, and set $A=g(Z_f)$. We now verify the
	fibre hypothesis of Lemma~\ref{lem:twoavoid}.
	
	Over $\overline k$, take $q\in Z_f$. If its $f$-fibre is
	$\ell_2+\ell_3$, then $q\in\ell_3$, so $g(q)=a_0$. Otherwise, write
	that $f$-fibre as $L+M$. Both lines are disjoint from $\ell_2$, since
	$\ell_2$ belongs to a different $f$-fibre. The plane containing $L$ and
	$M$ is the tangent plane $T_qX$. 
	
	If the $g$-fibre through $q$ is singular, choose a component $N$
	passing through $q$. Every line on $X$ passing through $q$ lies in
	$T_qX$. The section cut out by $T_qX$ on $X$ is
	$\ell_1+L+M$, so $N$ is one of $\ell_1,L,M$. Every component of a $g$-fibre
	meets $\ell_2$. Thus $N=\ell_1$, since $L+M$ is disjoint from $\ell_2$.
	The $g$-fibre is $\ell_1+\ell_3$.
	Consequently, every fibre of $g$ above $A$ is either smooth or the split fibre $\ell_1+\ell_3$. This verifies the
	hypotheses of Lemma~\ref{lem:twoavoid}, and the assertion follows from Proposition~\ref{prop:cubic}.
\end{proof}


\begin{prop}\label{prop:nonboundary}
Theorem~\ref{thm:main} holds for every hyperplane section $H\ne D$.
\end{prop}

\begin{proof}
The surface $X_m$ has three $k$-lines $D_x,D_y,D_z$, all lying in the hyperplane section $D$. Since \(H\neq D\), at least one line is not contained in $H$, and the assertion follows from Proposition~\ref{prop:twolines}.
\end{proof}

\section{The hyperplane section $H=D$}
\label{sec:vieta}

The pencil of planes through $D_x$
defines a conic bundle $f:X_m\to\PP^1_k$.
For finite $t$, the fibre is
\begin{equation*}
 C_t:\quad x=tw,\qquad
 by^2+cz^2-tyz+(at^2-m)w^2=0.
\end{equation*}
It is smooth if and only if
\begin{equation}\label{eq:smoothconic}
 (t^2-4bc)(at^2-m)\ne0.
\end{equation}
The fibre at infinity is
$ f^{-1}(\infty)=D_y+D_z.$


Let $T_0\supset \Omega_k^\infty$ be a finite set of places  containing all finite places at which one of $a,b,c$ is not a unit or $m$
is not integral. We fix the natural projective $\OO_{T_0}$-model $\sX_m$ of $X_m$ defined by \eqref{eq:projective}.

Set the three vertices of $D$ by
\[
 E=\{[1:0:0:0],[0:1:0:0],[0:0:1:0]\},\qquad
 V_m=X_m\setminus E.
\]
Let $\mathscr E$ be the closure of $E$ in $\sX_m$ and let
$\sV=\sX_m\setminus \mathscr E$. Recall that $U_m=X_m\setminus D$.

\begin{lem}\label{lem:vertices}
Let $T\supset T_0$ be a finite set of places and let $R_0\in U_m(k)$. Then there is a point $R\in U_m(k)$,
arbitrarily close to $R_0$ at every place of $T$, such that
\begin{equation}\label{eq:vertexavoidance}
 R\in\sV(\OO_v)\qquad\text{for every }v\notin T.
\end{equation}
\end{lem}

\begin{proof}

By Proposition~\ref{prop:ordinary}, we may assume without loss of generality that $t=x(R_0)$ satisfies
\eqref{eq:smoothconic}, i.e., that the fibre $C_t$ is smooth. Let $\widetilde f$ be the extension of $f$ to the natural projective
model over $\OO_{T_0}$.  Let
\[
 T_1=T\cup\{v\notin T:\bar t_v=\bar \infty\},
\]
which is a finite set. For every $v\in T_1\setminus T$, choose 
$\bar p_v=[1:0:1:0]$, a smooth $\FF_v$-point of the fibre at infinity outside
$E$. 
Applying Hensel's lemma for $C_t$, 
we may find a local point
$p_v\in C_t(k_v)$ reducing to $\bar p_v$.

The smooth conic $C_t$ contains $R_0$, so $C_t\simeq\PP^1_k$.
Weak approximation gives $R\in C_t(k)\cap U_m(k)$ arbitrarily
close to $R_0$ at $T$ and sufficiently close to $p_v$ at
$T_1\setminus T$. Hence we may assume
$\bar R_v\notin E$ for $v\in T_1\setminus T$.
For $v\notin T_1$, we have $\bar t\neq \bar \infty$, so $\bar R_v\notin E$
automatically, since the three vertex sections lie in the fibre at
infinity. This proves~\eqref{eq:vertexavoidance}.
\end{proof}


\subsection{Growth of boundary multiplicities}

For $v\notin T_0$,
primitive integral coordinates
$[\xi:\eta:\zeta:\omega]$ for $R\in X_m(k_v)$ give the intersection multiplicities
\begin{equation}\label{eq:min}
 \begin{aligned}
 n_v(\sD_x,R)&=\min\{\ord_v(\xi),\ord_v(\omega)\},\\
 n_v(\sD_y,R)&=\min\{\ord_v(\eta),\ord_v(\omega)\},\\
 n_v(\sD_z,R)&=\min\{\ord_v(\zeta),\ord_v(\omega)\}.
 \end{aligned}
\end{equation}


The three Vieta involutions of $U_m$ are
\begin{align*}
 \sigma_x(x,y,z)&=(yz/a-x,y,z),\\
 \sigma_y(x,y,z)&=(x,xz/b-y,z),\\
 \sigma_z(x,y,z)&=(x,y,xy/c-z).
\end{align*}
These polynomial involutions preserve \eqref{eq:affine}
by Vieta's formulas; for example,
\begin{equation*}
 ax(yz/a-x)=by^2+cz^2-m.
\end{equation*}
Put
\begin{equation*}
 \Phi=\sigma_z\circ\sigma_y\circ\sigma_x.
\end{equation*}

For a point $R\in U_m(k_v)$, we may define the pole-order triple refers to its affine coordinates of $U_m$. The pole order of a coordinate $\lambda$ is
$\max\{0,-\ord_v(\lambda)\}$ and the pole order of $0$ is defined to be zero.
\begin{lem}\label{lem:initialpoles}
Let $v\notin T_0$ and
$R\in U_m(k_v)\cap\sV(\OO_v)$. Then either $R$ has affine integral
coordinates, or its pole-order triple  is one of
\begin{equation*}
 (0,r,r),\qquad(r,0,r),\qquad(r,r,0),\qquad r\ge1.
\end{equation*}
After one application of $\Phi$, the corresponding triples are
\begin{equation}\label{eq:firstcycle}
 (2r,3r,5r),\qquad(r,2r,3r),\qquad(r,r,2r).
\end{equation}
\end{lem}

\begin{proof}
Choose primitive coordinates $[\xi:\eta:\zeta:\omega]$ for $R$,
with $\omega\ne0$. If $\omega$ is a unit, the affine coordinates are
integral. Otherwise, since $R\in \sV(\OO_v)$, the reduction of $R$ lies on exactly one side of the boundary triangle. If this side is $D_x$, then
$\eta,\zeta$ are units, and
\[
 \xi\eta\zeta=\omega(a\xi^2+b\eta^2+c\zeta^2-m\omega^2)
\]
implies $\ord_v(\xi)\ge\ord_v(\omega)$. Thus the pole orders are
$(0,r,r)$ for $r=\ord_v(\omega)\ge1$. The other two sides give the
remaining cases.

Write $R=(x,y,z)$ and put
\[
 x'=yz/a-x,\qquad y'=x'z/b-y,\qquad z'=x'y'/c-z.
\]
Then $\Phi(R)=(x',y',z')$. For $(0,r,r)$, each successive product
has strictly larger pole order, giving $(2r,3r,5r)$.

For $(r,0,r)$, the identity $axx'=by^2+cz^2-m$ has a unique term $cz^2$
of least valuation $-2r$. Thus
$\ord_v(x')=-r$. The products $x'z$ and $x'y'$ then have pole orders
$2r$ and $3r$, giving $(r,2r,3r)$.

For $(r,r,0)$, only $by^2$ has the least valuation $-2r$ in the same identity,
so $\ord_v(x')=-r$. At the point $(x',y,z)$, the unique term of least valuation
$-2r$ in the identity
\[
 byy'=ax'^2+cz^2-m
\]
is $ax'^2$, so $\ord_v(y')=-r$. Since $z$ is integral, the equality $z'=x'y'/c-z$ implies that $z'$ has
pole order $2r$. This gives $(r,r,2r)$ and proves
\eqref{eq:firstcycle}.
\end{proof}

\begin{lem}\label{lem:growth}
Let $v\notin T_0$. If $R\in U_m(k_v)$
has pole orders $(A,B,A+B)$ with $0<A\le B$, then:
\begin{enumerate}
	\item the boundary
	multiplicities are
	\begin{equation}\label{eq:orders}
		(n_v(\sD_x,R),n_v(\sD_y,R),n_v(\sD_z,R))=(B,A,0);
	\end{equation}
	\item $\Phi(R)$ has pole orders
	\begin{equation}\label{eq:recurrence}
		(A+2B,\,2A+3B,\,3A+5B).
	\end{equation}
\end{enumerate}
\end{lem}

\begin{proof}
	Write $R=(x,y,z)$. Multiplying $[x:y:z:1]$ by $\pi_v^{A+B}$ gives primitive coordinates
	with valuations $(B,A,0,A+B)$. Formula~\eqref{eq:min} therefore gives
	\eqref{eq:orders}.
	
Since $0<A\le B$, each successive product
has strictly larger pole order. A direct computation proves
\eqref{eq:recurrence}. 
\end{proof}

\begin{prop}\label{prop:amplification}
Let $n\ge 2$ be an integer. Let $v\notin T_0$ and let $\sX_m$ be
 the natural projective $\OO_{T_0}$-model.  Choose $N\ge 1$ such that
$3^{N-1}\ge n$, and let $\varphi=\Phi^N$. Then
\[
 \varphi\bigl(U_m(k_v)\cap\sV(\OO_v)\bigr)
 \subset\CC_v(\sX_m,(1-1/n)D).
\]
\end{prop}

\begin{proof}
Let $R\in U_m(k_v)\cap\sV(\OO_v)$. If its affine coordinates are
integral, they remain integral under the Vieta involutions, so all
boundary multiplicities of $\varphi(R)$ are zero. Otherwise,
Lemma~\ref{lem:initialpoles} gives pole orders $(A,B,A+B)$ after the
first application of $\Phi$, with $1\leq A\leq B$.
Lemma~\ref{lem:growth} preserves this form and at least triples the
smallest entry with each further application of $\Phi$. After $N$ applications, this entry is therefore at
least $3^{N-1}A\ge n$. By \eqref{eq:orders}, the two nonzero boundary
multiplicities are then at least $n$, proving the Campana
conditions.
\end{proof}


\subsection{Proof of the theorem}
\label{sec:proof}

\begin{proof}[Proof of Theorem~\ref{thm:main}]
For $H\ne D$, apply Proposition~\ref{prop:nonboundary}. Now assume
that $H=D$. Let $T_0$ and the $\OO_{T_0}$-model $\sX_m$ be as above. 

Fix $\varphi=\Phi^N$ as in Proposition~\ref{prop:amplification}.
For an adelic Campana point
$(P_v)$ orthogonal to $\Br(X_m)$, choose an arbitrary open neighbourhood $\prod_{v\in T}\mathcal N_v\times\prod_{v\notin T}\CC_v(\sX_m,(1-1/n)D)$ of $(P_v)$ for some $T\supset T_0$.
Proposition~\ref{prop:ordinary} gives a rational point $A$ such that $A\in\mathcal N_v$ at every $v\in T$ and
$R_0=\varphi^{-1}(A)\in U_m(k)$.


By the continuity of $\varphi$, we may choose neighbourhoods of $R_0$ at
the places in $T$ such that their images lie in $\mathcal N_v$.
Lemma~\ref{lem:vertices} gives a point
$R\in U_m(k)$ lying in these neighbourhoods such that $R\in\sV(\OO_v)$ for
every $v\notin T$. Put $P=\varphi(R)$. At each $v\in T$, $P$ lies in
$\mathcal N_v$.
Proposition~\ref{prop:amplification} gives the remaining Campana
conditions for $v\notin T$,
which completes the proof.
\end{proof}

\begin{proof}[Proof of Corollary~\ref{cor:boundary}]
For $M=\max\{n_x,n_y,n_z\}$, one has
$(1-1/M)D\ge\Delta_{\boldsymbol n}$.
Since $\Br(X_m)/\Br_0(X_m)$ is finite, the corollary then follows from
Theorem~\ref{thm:main} with weight $M$ and
Lemma~\ref{lem:monotonicity}\,(2).
\end{proof}

%


\begin{rmk}\label{rem:finiteweights}
The approximation theorem  for generalized Markoff surfaces proved above concerns finite weights. 
Both Vieta iterations and the rational-conic
arguments are similarly restricted to finite weights.
These methods do not give approximation results for infinite weights.
Indeed, the failures of integral approximation recalled in the
introduction show that such an extension is impossible in general.
\end{rmk}

\needspace{8\baselineskip}

\end{document}